\documentclass[12pt,letterpaper]{amsart}
\usepackage[pass]{geometry}
\usepackage{amssymb,amsmath,amsthm,amsgen}
\usepackage{tikz}
\usepackage{comment}
\usepackage{}
\usepackage{enumitem}

\newcommand{\old}[1]{}
\theoremstyle{plain}
\newtheorem{thm}{Theorem}[section]
\newtheorem{lem}[thm]{Lemma}

\newtheorem{cor}[thm]{Corollary}

\newtheorem{prop}[thm]{Proposition}
\theoremstyle{definition}
\newtheorem{defn}[thm]{Definition}
\newtheorem{remark}[thm]{Remark}

\numberwithin{equation}{section}

\DeclareMathOperator{\arctanh}{arctanh}

 at 10truept

\title{The plethystic inverse of the odd Lie representations} 

\author{Sheila Sundaram}
\address{Pierrepont School, One Sylvan Road North, Westport, CT 06880}
\email{shsund@comcast.net}
\date{2021 Dec 8}
\dedicatory{In memory of my mother, Nirmala Sundaram}

\commby{Patricia Hersh}


\begin{document}
\begin{abstract}  The Frobenius characteristic of  $Lie_n,$ the  representation of the symmetric group $S_n$ afforded by the multilinear part of the free Lie algebra, is known to satisfy many interesting plethystic identities.  In this paper we prove a conjecture of Richard Stanley establishing the plethystic inverse of the sum $\sum_{n\geq 0} Lie_{2n+1}$ of the odd Lie characteristics. We obtain an apparently new plethystic decomposition of the regular representation of $S_n$ in terms of irreducibles indexed by hooks, and the Lie representations.  We determine the plethystic inverse of the alternating sum of the odd Lie characteristics.

\end{abstract}

\maketitle

\section{Introduction}

  Let $Lie_n$ denote the Frobenius characteristic of the representation of the symmetric group $S_n$ obtained by inducing the representation afforded by a primitive $n$th root of unity from the cyclic subgroup $C_n$ (generated by the long $n$-cycle) up to $S_n.$  The symmetric functions $Lie_n$ are known to arise in many different contexts, often involving plethystic identities.  See, e.g. \cite[Solutions to Ex. 7.88-7.89]{St4EC2} and \cite{Su1}, \cite{Su2}.  In particular $Lie_n$ describes the $S_n$-action on the multilinear component of the free Lie algebra on $n$ generators.

Let $e_n$ denote the elementary symmetric function of degree $n$; it is the Frobenius characteristic of the sign representation of $S_n.$
In this paper we establish a plethystic identity conjectured by Richard Stanley, Theorem~\ref{OrigRPS} below.   
\begin{thm}\label{OrigRPS} 
\[\text{The symmetric functions }\dfrac{\sum_{n\geq 0} e_{2n+1}}{\sum_{n\geq 0} e_{2n}}  \text{ and }
\sum_{n\geq 0} Lie_{2n+1} \text{ are plethystic inverses.}\]

\end{thm}

We also show that 
\begin{thm}\label{AltRPS} 
\[ \text{The functions }  \dfrac{\sum_{n\geq 0} (-1)^{n} e_{2n+1}}{\sum_{n\geq 0} (-1)^{n} e_{2n}} \ \text{and}\ 
 \sum_{n\geq 0} (-1)^{n} Lie_{2n+1}\ \text{are plethystic inverses.}\]
\end{thm}

  The quotient of symmetric functions in Theorem~\ref{AltRPS} has been studied by other authors \cite{Car, StAltEnum}.  The  odd Lie representations also appear in \cite{CHS} in connection with the free Jordan algebra.
  
\section{Preliminary identities}

We follow \cite{M} and \cite{St4EC2} for notation regarding symmetric functions.

In particular, $h_n,$ $e_n$ and $p_n$ denote respectively the complete homogeneous, elementary and power-sum symmetric functions, and   $\omega$ is the involution on the ring of symmetric functions which takes $h_n$ to $e_n.$  
If $\mathrm{ch}$ is the Frobenius characteristic map from the representation ring of the symmetric group $S_n$ to the ring of symmetric functions with real coefficients, then 
$h_n=\mathrm{ch}(1_{S_n})$ is the characteristic of the trivial representation, and $e_n=\mathrm{ch}({\rm sgn}_{S_n})$ is the characteristic of the sign representation of $S_n.$   Finally, for a partition  $\mu$  of $n$, the irreducible $S_n$-module indexed by $\mu$ maps to the Schur function $s_\mu$ under the map $\mathrm{ch}.$

If $q$ and $r$ are characteristics of representations of $S_m$ and $S_n$ respectively, they yield a representation of the wreath product $S_m[S_n]$ in a natural way, with the property that when this representation is induced up to $S_{mn},$ its Frobenius characteristic is the plethysm of $q$ with $r,$ denoted $q[r].$ For more background about this operation, see \cite{M}.  We list the following key properties, in particular the fact that plethysm $(\cdot)[r]$ with a fixed symmetric function $r$ is an endomorphism (in the first argument) of the ring of symmetric functions \cite[(8.3)]{M}.  See also \cite[Chapter 7, Appendix 2, A2.6]{St4EC2}.

\begin{prop}\label{PlethProperties} If $q, r$ are symmetric functions of homogeneous degrees, 
$f,g, Y$ are arbitrary symmetric functions, and $\lambda$ is any partition, then 
\begin{enumerate}
\item $c[q]=c$ for any constant $c$.
\item $p_n[p_m]=p_{nm}=p_m[p_n]$. 
\item $p_n[f]=f[p_n]$.
\item $(fg)[q]=f[q]\cdot g[q];$ in particular $(cg)[q]=c\cdot  (g[q])$ for any constant $c$.
\item If $fg=X,$ then $(\frac{f}{X})[q]=g[q]=\frac{f[q]}{X[q]}.$ 
In particular $(\frac{1}{f})[q]=\frac{1}{f[q]}.$
\item  $f([g[Y]]) =( f[g])[Y]$, i.e.  plethysm is associative.
\item $s_\lambda [q+r]=\sum_{\mu\subseteq \lambda} s_{\lambda/\mu} [q] s_\mu[r].$ 

Here we single out the  special cases  $\lambda=(n), \lambda=(1^n):$

$h_n[q+r]=\sum_{k=0}^n h_k[q] h_{n-k}[r]$ and 
$e_n[q+r]=\sum_{k=0}^n e_k[q] e_{n-k}[r]$.
\item $q[-r]= (-1)^{\text{deg }q}  (\omega q)[r].$
\item $\omega(q[r])=\left( \omega^{\text{deg }r}(q)\right)[\omega r].$
\item $f[g]=p_1\iff g[f]=p_1.$
\end{enumerate}
\end{prop}

Define \begin{align}\label{defHE} &H(t)=\sum_{i\geq 0}t^i  h_i, \quad E(t) = \sum_{i\geq 0} t^i  e_i; \\
&H=\sum_{i\geq 0}  h_i, \quad E= \sum_{i\geq 0}  e_i. 
\end{align} 

Recall the following well-known facts about the series $H(t)
$ and $E(t).$  Parts (3) and (4) are  immediate consequences of (1) and (2), and Part (5) follows from the Pieri rule.

\begin{prop}\label{FundamentalHE} (\cite{M}, \cite{St4EC2}) \begin{enumerate}
\item $H(t)=\exp \left(\sum_{k\ge 1}{t^kp_k/k}\right);$
\item $E(t)=\exp \left(\sum_{k\ge 1}(-1)^{k-1}{t^kp_k}/{k}\right);$
\item $H(t)E(-t)=1;$
\item $H(t)E(t)=\exp\left(\sum_{k\ge 1, k\mathrm{\, odd}}  {2t^kp_k}/{k}\right);$
\item $H(t)E(t)=1+2\sum_{n\ge 1} t^n Hk_n,$ where we have written $Hk_n$ for the sum of Schur functions corresponding to hook shapes $\sum_{k=0}^{n -1} s_{(n-k, 1^k)}.$
\end{enumerate}

\end{prop}

We refer the interested reader to \cite{R} for background about the free Lie algebra, although our use of the Lie representation requires no knowledge other than what is presented here.

Define $Lie_n$ to be the Frobenius characteristic of the representation of $S_n$ afforded by the multilinear component of the free Lie algebra on $n$ generators.  It is well known that 
\begin{equation}\label{BrandtLie} 
 Lie_n=\frac{1}{n}\sum_{d|n}\mu(d) p_d^{\frac{n}{d}},
\end{equation}
where $\mu$ is the number-theoretic M\"obius function.  This reflects the fact that $Lie_n$ is the Frobenius characteristic of the $S_n$-module obtained by inducing a faithful irreducible representation of the cyclic subgroup $C_n$ (generated by the long $n$-cycle) up to $S_n.$

In \cite[Theorem~3.2]{Su1}, \cite[Theorem~5.8]{Su2} it was shown that the following plethystic identities satisfied by $Lie_n$ are equivalent, and a uniform derivation of the identities was provided.  The first identity is a theorem of Thrall \cite{T}.   The equivalence of the second identity with the first is a consequence of \cite[Proposition~6.6]{Su2}.  See also \cite{Su3}.

\begin{thm} \label{ThrallHE} The following identities hold:
\begin{enumerate}[itemsep=.1in]
\item  $H[Lie(t)]=\dfrac{1}{1-tp_1}.$
\item $E[Lie(t)] =\dfrac{1-t^2p_2}{1-tp_1}.$
\end{enumerate}
\end{thm}

\begin{defn}\label{ParityNotn} If $F=\sum_{n\geq 0} f_n$ is any formal power series of symmetric functions $\{f_n\}$ where $f_n$ is homogeneous of degree $n,$  we will write 
\[F_{odd}=f_1+f_3+\dots=\sum_{n\ge 0} f_{2n+1}, \quad
F_{even}=f_0+f_2+f_4+\cdots=\sum_{n\geq 0} f_{2n}.\]
We also write, for an indeterminate $t,$ $F(t)=\sum_{n\geq 0} t^n f_n,$ and 
\[F_{odd}(t)=tf_1+t^3f_3+\cdots=\sum_{n\ge 0}t^{2n+1} f_{2n+1}, \quad
F_{even}(t)=f_0+t^2f_2+t^4f_4+\cdots=\sum_{n\geq 0}t^{2n} f_{2n}.\]
\end{defn}

The following observations will be useful for our arguments.

\begin{prop}\label{myHE1}  We have 
\begin{enumerate}
\item $H_{odd}(t) E_{even}(t)-H_{even}(t) E_{odd}(t)=0.$
\item $H_{even}(t) E_{even}(t)-H_{odd}(t)E_{odd}(t)=1.$
\item $H_{odd}(t)=\dfrac{H(t) E(t) -1}{2E(t)},$ 
$H_{even}(t)=\dfrac{H(t) E(t) +1}{2E(t)}.$ 
\end{enumerate}
\end{prop}

\begin{proof}   Parts (1) and (2) follow 
from Proposition~\ref{FundamentalHE} by isolating the terms of even and odd degree of the power series in $H(t)E(-t)=1$.  Part (3) follows from the fact that 
\[2 H_{odd}(t) =H(t)-H(-t)=H(t)-\dfrac{1}{E(t)},\] 
with the corresponding statement for $H_{even}(t).$   
\end{proof}

\begin{prop}\label{myHE1a}  The symmetric function $\dfrac{H_{odd}(t)}{H_{even}(t)}=\dfrac{E_{odd}(t)}{E_{even}(t)}$ is invariant under the involution $\omega$, and is equal to each of the following:
\begin{enumerate}
\item $\dfrac{H(t) E(t)-1}{H(t)E(t)+1};$

\item $\dfrac{\sum_{n\ge 1}t^nH\!k_n}{1+\sum_{n\ge 1}t^n H\!k_n},$ 
\noindent
where $H\!k_n$ is the sum of the $n$ Schur functions corresponding to  hook shapes of size $n$

\item $\tanh\left(\sum_{k\ge 0}  \dfrac {t^{2k+1}p_{2k+1}}{2k+1} \right).$ This in turn equals 
$\sum_{n\ge 0} (-1)^n E_{2n+1} \frac{Z^{2n+1}}{(2n+1)!},$ where we have written 
$Z=\sum_{k\ge 0}  \dfrac {t^{2k+1}p_{2k+1}}{2k+1}$, and $E_{2n+1}$ is the tangent number.
\end{enumerate}
\end{prop}
\begin{proof} Parts (1) and (2) are  clear from Part (3) of Proposition~\ref{myHE1} and Part (5) of Proposition~\ref{FundamentalHE}.  

For Part (3), we use Part (4) of Proposition~\ref{FundamentalHE}.  Writing $\theta=\sum_{\stackrel{k\ge 1}{ k\text{ odd}}  } \dfrac{t^kp_k}{k},$ the function in Part (1) equals 
$\dfrac{e^{2\theta}-1}{e^{2\theta}+1}=\tanh \theta,$
as claimed.  The invariance under $\omega$ is  clear. 
The last statement is a consequence of the generating function  \cite{St4EC1} for the tangent numbers $E_{2n+1}$.
\end{proof}

\section{Proofs of Theorem~\ref{OrigRPS} and Theorem~\ref{AltRPS}}

  The crux of the proof of Theorem~\ref{OrigRPS} lies in the following:
\begin{prop}\label{pleth-arctanh}  We have the plethystic identity
\begin{equation}\label{pleth-arctanhEqn}\sum_{k\, {\text odd}}  
\dfrac {p_{k}}{k}[Lie_{odd}]=\sum_{m\, {\text odd}} \dfrac{p_1^m}{m}. \end{equation}
\end{prop}
\begin{proof} Let $\delta_{1,j}$ denote the Kronecker $\delta,$ which is 1 if $j=1$ and zero otherwise.
Using  the properties listed in Proposition~\ref{PlethProperties}, in particular Parts (2) and (3), we obtain \begin{align*}&\sum_{k\ge 1, k\text{ odd}}\frac{p_k}{k}[ \sum_{n\text{ odd}} \frac{1}{n} \sum_{d|n} \mu(d) p_d^{\tfrac{n}{d}}]
=\sum_{k\ge 1, k\text{ odd}}\frac{1}{k}\, \sum_{n\text{ odd}} \frac{1}{n} \sum_{d|n} \mu(d) p_d^{\tfrac{n}{d}}[p_k]\\
&=\sum_{\stackrel{k,n\ge 1}{ k,n\text{ odd}}}\frac{1}{kn}\,   \sum_{d|n} \mu(d) p_{kd}^{\tfrac{n}{d}}=2\sum_{\stackrel{k,m,d\ge 1}{ k,m,d\text{ odd}}}\frac{1}{kmd}\,   \mu(d) p_{kd}^m\\
&\text{ where we have put }n=md;\\
&=\sum_{\stackrel{m,j\ge 1}{ m,j\text{ odd}}}\frac{1}{jm}\, \sum_{d|j}  \mu(d) p_{j}^m   \text{ where we have put }j=kd;\\
&=\sum_{\stackrel{m,j\ge 1}{ m,j\text{ odd}}}\frac{p_j^m}{jm}\, \sum_{d|j}  \mu(d) =\sum_{\stackrel{m\ge 1}{ m\text{ odd}}}\frac{p_1^m}{m}, \text{ since } \sum_{d|j}  \mu(d) =\delta_{1,j}\\
&=\frac{1}{2}\log \frac{1+p_1}{1-p_1}.
\end{align*}
\end{proof}

\begin{proof}[Proof of Theorem~\ref{OrigRPS}]
  Take  the plethysm of  Part (3) of Proposition~\ref{myHE1a} (for $t=1$) with $Lie_{odd}$.  Associativity of the plethysm operation gives 
\[\tanh\left(\sum_{k\ge 0}  \dfrac {p_{2k+1}}{2k+1} \right)[Lie_{odd}] =\tanh\left(\sum_{m\, {\text odd}} \dfrac{p_1^m}{m}   \right).\]
The right-hand side of the identity in Proposition~\ref{pleth-arctanh} is simply $\arctanh{p_1}$ as a formal power series.  Hence the above calculation reduces to 
$\tanh(\arctanh(p_1))=p_1,$ as claimed.
\end{proof}





We now present  a restatement of Theorem~\ref{OrigRPS} that is of independent interest.

\begin{lem}\label{lem:myRPSrestated} The plethystic identity of Theorem~\ref{OrigRPS} is equiivalent to the following identity:
\begin{equation}\label{myRPSrestated} 
(HE)[Lie_{odd}]=\dfrac{1+p_1}{1-p_1}.
\end{equation}
\end{lem}
\begin{proof}
Using  the fact that plethysm with $g$ is a ring homomorphism,  specifically Proposition~\ref{PlethProperties}, Parts (4) and (5), in conjunction with Part (1) of Proposition~\ref{myHE1a}, we may restate Theorem~\ref{OrigRPS} as 
\begin{equation}\label{RPSrestated}
\dfrac{E_{odd}[Lie_{odd}]}{E_{even} [Lie_{odd}]}=\dfrac{(HE)[Lie_{odd}]-1}{(HE) [Lie_{odd}]+1}=p_1.
\end{equation}
Rearranging, we see that Theorem 1.1 is equivalent to 

$\qquad\qquad\qquad(H E)[Lie_{odd}]-1=p_1\cdot \left((HE)[Lie_{odd}]+1\right), $
giving Eqn.~\eqref{myRPSrestated}.
\end{proof}

Identity~\eqref{myRPSrestated} leads to the following  reformulation of Theorem~\ref{OrigRPS}; compare with Thrall's decomposition of the regular representation in 
 Theorem~\ref{ThrallHE}.

\begin{prop}\label{HookReg} The  regular representation admits the  plethystic decomposition 
\[p_1^n=H\!k[Lie_{odd}]|_{\mathrm{deg\,}n},\]
where $H\!k$ is the sum of all hooks,  
$H\!k=s_{(1)}+\sum_{n\ge 2}\sum_{r=0}^{n-1} s_{(n-r, 1^r)}.$
\end{prop}

\begin{proof} The sum of hooks enters the left-hand side of Eqn.~\eqref{myRPSrestated} via Part (5) of Proposition~\ref{FundamentalHE}, giving 
\[1+2\ H\!k[Lie_{odd}]= 1+2\sum_{n\ge 1} p_1^n,\]
which is as claimed.
\end{proof}

It would be interesting to find a representation-theoretic interpretation of the above decomposition.

The following lemma  is easily deduced from Proposition~\ref{PlethProperties}, Part (7), or from Proposition~\ref{FundamentalHE}, Part (1); see also \cite{Su1}. 
\begin{lem}\label{HESum} For any series of symmetric functions $F$ and $ G$,
$H[F+G]=H[F]\cdot H[G],$  $E[F+G]=E[F]\cdot E[G],$ 
and thus 
$(H\!E)[F+G]=(H\!E)[F]\cdot (H\!E)[G].$
\end{lem}

From identity~\eqref{myRPSrestated}
 we also see that 
\begin{prop} \[(H\!E)[Lie_{even}]=(1-p_2)(1-p_1)^{-2}=(1-p_1)^{-1}E[Lie].\]
\end{prop}
\begin{proof} Writing $Lie=Lie_{odd}+Lie_{even}$ and using Lemma~\ref{HESum} and Theorem~\ref{ThrallHE}, we have the identities
\begin{equation*} H[Lie_{odd}]\cdot H[Lie_{even}]=(1-p_1)^{-1},
\end{equation*}
\begin{equation*} E[Lie_{odd}]\cdot E[Lie_{even}]=(1-p_2)(1-p_1)^{-1}
\end{equation*}
Write $A=(H\!E)[Lie_{odd}],$ $B=(H\!E)[Lie_{even}].$  

Using the fact that plethysm  is a ring homomorphism in the first argument, we have 
\begin{equation}\label{HEPlethSplit} A\cdot B=
(H\!E)[Lie_{odd}]\cdot (H\!E)[Lie_{even}]=(1-p_2)(1-p_1)^{-2}.
\end{equation}
Since  $A=\dfrac{1+p_1}{1-p_1}$  from Eqn.~\eqref{myRPSrestated},
 the result follows.  The second equality is a consequence of Part (2) of Theorem~\ref{ThrallHE}.
\end{proof}

Next we address Theorem~\ref{AltRPS}.  For this we need to establish analogues of Propositions~\ref{myHE1a} and ~\ref{pleth-arctanh}.
In order to prove the identities below, we  work with the ring of symmetric functions with  coefficients  in $\mathbb{C}[t].$ 

We have the following analogues of Propositions~\ref{myHE1} and ~\ref{myHE1a}.  Recall that we write $H_{odd}^{alt}$ for the alternating sum $\sum_{k\ge 0} (-1)^k h_{2k+1},$ and $H_{even}^{alt}$ for the alternating sum $\sum_{k\ge 0} (-1)^k e_{2k},$  and similarly for $E.$ 
\begin{prop}\label{myHE2}  We have 
\begin{enumerate}
\item $H(it)=H_{even}^{alt}(t) +i H_{odd}^{alt}(t);$
\item $E(it)=E_{even}^{alt}(t) +i E_{odd}^{alt}(t);$
\item $H_{odd}^{alt}(t) E_{even}^{alt}(t)-H_{even}^{alt}(t) E_{odd}^{alt}(t)=0.$  In particular we have 

$(1-h_2+h_4-\ldots)(e_1-e_3+e_5-\ldots) =(h_1-h_3+h_5-\ldots)(1-e_2+e_4-\ldots).$
\item $H_{even}^{alt}(t) E_{even}^{alt}(t)+H_{odd}^{alt}(t)E_{odd}^{alt}(t)=1.$  In particular, we have

$(1-h_2+h_4-\ldots)(1-e_2+e_4-\dots)=1-(h_1-h_3+h_5-\ldots)(e_1-e_3+e_5-\ldots).$ 
\end{enumerate}
\end{prop}
\begin{proof}  This follows easily as in the proof of Proposition~\ref{myHE1}, since $H(it)E(-it)=1.$
\end{proof}

We will need the following expressions involving the alternating sum of the hooks $H\!k_n$ defined previously. 
Let \[H\!k_{even}^{alt}(t)=\sum_{n\ge 0, n\text{ even}} (-1)^{\frac{n}{2}}t^n H\!k_n,\quad  H\!k_{odd}^{alt}(t)=\sum_{n\ge 1, n\text{ odd}} (-1)^{\frac{n-1}{2}} t^n H\!k_n.\]
\begin{prop}\label{myHE2a} The symmetric function $\dfrac{H_{odd}^{alt}(t)}{H_{even}^{alt}(t)}=\dfrac{E_{odd}^{alt}(t)}{E_{even}^{alt}(t)}$ is invariant under the involution $\omega,$ and equals each of the following.
\begin{enumerate}[itemsep=.1in]
\item $\dfrac{H(it) E(it)-1}{i(H(it)E(it)+1)};$ 
\item $\dfrac{H\!k_{odd}^{alt}(t)}{ (H\!k_{even}^{alt}(t))^2+  
 (H\!k_{odd}^{alt}(t))^2};$ 
 \vskip.1in
 
 furthermore  $ (H\!k_{odd}^{alt}(t))^2=H\!k_{even}^{alt}(t)-(H\!k_{even}^{alt}(t))^2.$
\item $\tan \left( \sum_{k\ge 0}  \dfrac {(-1)^k t^{2k+1}p_{2k+1}}{2k+1} \right).$   
This in turn equals 
$\sum_{n\ge 0}  E_{2n+1} \frac{W^{2n+1}}{(2n+1)!},$ where we have written 
$W=\sum_{k\ge 0}  \dfrac {(-1)^k t^{2k+1}p_{2k+1}}{2k+1}$, and $E_{2n+1}$ is the tangent number. 
\end{enumerate}
\end{prop}
\begin{proof} The equality of the two quotients is immediate from Part (3) of the preceding proposition, and the invariance under $\omega$ then follows.

We have  
$2iH_{odd}^{alt}(t)=H(it)-H(-it),\quad  2H_{even}^{alt}(t)=H(it)+H(-it).$
Since $H(-it) E(it)=1,$ we obtain the expression in Part (1).

For simplicity put $X=Hk_{even}^{alt}(t), Y=Hk_{odd}^{alt}(t).$ 
From Part (5) of Proposition~\ref{FundamentalHE}, we see that 
\begin{equation}\label{eqn:Prop3.7Part2}(H\!E)(i)-1=2(X+iY-1), (H\!E)(i)+1=2(X+iY),\end{equation}
and hence the expression in Part (1) is converted to 
\[-i(1-(X+iY)^{-1})= \dfrac{Y}{X^2+Y^2}-i\dfrac{X^2+Y^2-X}{X^2+Y^2}.\]
The imaginary part is zero (since $\dfrac{E_{odd}^{alt}(t)}{E_{even}^{alt}(t)}$ has real coefficients), yielding Part (2).

Using Part (4) of Proposition~\ref{FundamentalHE}, we have, writing 
$\theta=\sum_{k\ge 1, k\text{ odd}} \frac{(-1)^{\frac{k-1}{2}}t^kp_k}{k},$
\[(H\!E)(it)=H(it)E(it)=\exp \sum_{\stackrel{k\ge 1}{ k\text{ odd}}} \frac{2i^kt^kp_k}{k}=e^{2i\theta},\]
Part (3) now follows, since  the quotient in Part (1)  equals 
$\dfrac{e^{2i\theta}-1}{i(e^{2i\theta}+1)}=\tan(\theta).$ The last statement is a consequence of  the generating function for the tangent numbers $E_{2n+1}$, \cite{St4EC1}.
\end{proof}

We are now ready to prove Theorem~\ref{AltRPS}. The analogue of Proposition~\ref{pleth-arctanh} is:

\begin{prop}\label{pleth-arctgt} The following plethystic identity holds:
\begin{equation}\label{pleth-arctgtEqn}\sum_{k\,{\text odd}}  (-1)^{\frac{k-1}{2}}\dfrac {p_{k}}{k}[Lie_{odd}^{alt}]
=\sum_{m\,{\text odd}}(-1)^{\frac{m-1}{2}}\dfrac{p_1^m}{m}.
\end{equation}
\end{prop}
\begin{proof} The left-hand side of \eqref{pleth-arctgtEqn} is
\begin{align*}
& \sum_{k\ge 1, k\text{ odd}}(-1)^{\frac{k-1}{2}}\frac{p_k}{k}[ \sum_{n\text{ odd}}(-1)^{\frac{n-1}{2}} \frac{1}{n} \sum_{d|n} \mu(d) p_d^{\tfrac{n}{d}}]\\
&=\sum_{k\ge 1, k\text{ odd}} (-1)^{\frac{k-1}{2}}\frac{1}{k}\, \sum_{n\text{ odd}} (-1)^{\frac{n-1}{2}}\frac{1}{n} \sum_{d|n} \mu(d) p_d^{\tfrac{n}{d}}[p_k]\\
&\sum_{k,n\,{\text odd}} \frac{1}{kn} (-1)^{\frac{k-1}{2}}(-1)^{\frac{n-1}{2}}\sum_{d|n}\mu(d) p_{kd}^{\frac{n}{d}}=\sum_{\stackrel{k,m,d\ge 1}{ k,m,d\text{ odd}}}\frac{1}{kmd}\,  (-1)^{\frac{k-1}{2}}(-1)^{\frac{md-1}{2}} \mu(d) p_{kd}^m\\
&\text{ where we have put }n=md;\\
&=\sum_{\stackrel{m,j\ge 1}{ m,j\text{ odd}}}\frac{1}{jm}\, \sum_{d|j} 
(-1)^{\frac{j/d-1}{2}}(-1)^{\frac{md-1}{2}} \mu(d) p_{j}^m   \text{ where we have put }j=kd;\\
&=\sum_{\stackrel{m,j\ge 1}{ m,j\text{ odd}}}\frac{p_j^m}{jm}\, \sum_{d|j}  (-1)^{\frac{j/d+md}{2}}(-1)\mu(d); \text{ note that } j/d+md \text{ is even}.
\end{align*}
Since $d$ is always odd, we have $(-1)=(-1)^d$ and the inner sum can be rewritten as 
\begin{align*}
& \sum_{d|j}  (-1)^{\frac{j+md^2}{2}}(-1)\mu(d)
=\sum_{d|j}  (-1)^{\frac{(j-1)+(md^2-1)}{2}}\mu(d)\\ 
&=(-1)^{\frac{j-1}{2}}(-1)^{\frac{m-1}{2}}\sum_{d|j}  \mu(d) (-1)^{\frac{m(d^2-1)}{2}}=(-1)^{\frac{j-1}{2}}(-1)^{\frac{m-1}{2}}\sum_{d|j}  \mu(d),\\
&\text{since, for } d \text{ odd, }  d^2-1=(d-1)(d+1)\equiv 0\bmod 4.
\end{align*}
But $\sum_{d|j}  \mu(d)$ is 1 if $j=1$ and 0 otherwise, making
 the left hand side of ~\eqref{pleth-arctgtEqn} equal to 
\[\sum_{\stackrel{m\ge 1}{ m\text{ odd}}}(-1)^{\frac{m-1}{2}} \frac{p_1^m}{m}\]
as claimed.
\end{proof}

\begin{proof}[Proof of Theorem~\ref{AltRPS}]  
 Part (3) of Proposition~\ref{myHE2a} gives, as before by associativity, 
\[\dfrac{E_{odd}^{alt}}{E_{even}^{alt}}[Lie_{odd}^{alt}]
=\tan \left( \sum_{k\text{ odd}}  
(-1)^{\frac{k-1}{2}}\dfrac{p_k}{k}[Lie_{odd}^{alt}] \right)
=\tan\left(\sum_{\stackrel{m\ge 1}{ m\text{ odd}}}(-1)^{\frac{m-1}{2}} \frac{p_1^m}{m}\right).\]
Clearly the argument of the tangent function is  $\arctan p_1,$ so the plethysm equals $p_1.$
\end{proof}

Exactly as in the previous case, Theorem~\ref{AltRPS} has an equivalent formulation which also gives an analogue of Proposition~\ref{HookReg}.

\begin{prop} We have the identities
\begin{equation}\label{HookRegAltEven} \sum_{m\ge 2,\, m\text{ even}} (-1)^{\tfrac{m}{2}-1}H\!k_m[Lie^{alt}_{odd}]\vert_{{\rm deg\,} 2n}=(-1)^n p_1^{2n}, n\ge 1,
\end{equation}
and 
\begin{equation}\label{HookRegAltOdd} \sum_{m\ge 1,\, m\text{ odd}} (-1)^{\tfrac{m-1}{2}}H\!k_m[Lie^{alt}_{odd}]\vert_{{\rm deg\,} 2n+1}=(-1)^n p_1^{2n+1}, n\ge 0.
\end{equation}
\end{prop}
\begin{proof}
Using Part (1) of Proposition~\ref{myHE2a} and rearranging the plethysm, Theorem 1.2 
has the equivalent formulation
\begin{equation*}
(H\!E)(i)[Lie_{odd}^{alt}]-1=i p_1\cdot \left((H\!E(i))[Lie_{odd}^{alt}]+1\right),
\end{equation*} or 
\begin{equation}\label{myAltRPSrestated}
(H\! E)(i)[Lie_{odd}^{alt}]=\dfrac{1+ip_1}{1-ip_1}.
\end{equation}
In the proof of Proposition~\ref{myHE2a}, put $t=1$ in  Eqn.~\eqref{eqn:Prop3.7Part2}. We obtain 
\[H\!E(i)=(2X'-1)+2iY'\] for $X'=H\!k_{even}^{alt}(1)=\sum_{m\ge 0,\, m\text{ even}} (-1)^{\tfrac{m}{2}} H\!k_m, Y'=H\!k_{odd}^{alt}(1)=\sum_{m\ge 1,\, m\text{ odd}} (-1)^{\tfrac{m-1}{2}} H\!k_m$,   and hence we have the identities
\begin{equation*}\label{evenHook-oddHook-of-oddLie-alt}
(2X'-1)[Lie_{odd}^{alt}]=1-\dfrac{2p_1^2}{1+p_1^2}, \quad (2Y')[Lie_{odd}^{alt}]=\dfrac{2p_1}{1+p_1^2}.
\end{equation*}
Noting that 
 $2X'-1=1+2\sum_{m\ge 2,\, m\text{ even}} (-1)^{\tfrac{m}{2}} H\!k_m,$  the result follows.
\end{proof}

As mentioned in the Introduction, the quotient of symmetric functions
 $ \dfrac{E_{odd}^{alt}}{E_{even}^{alt}}$ has already appeared in the literature. The earliest reference known to us, albeit implicitly, is a 1973 paper of  Carlitz \cite{Car}.   See also \cite[Solution to Exercise 7.64 (c)]{St4EC2}, \cite[Theorem 2.2 and  the note on Page 4]{StAltEnum} and \cite[Proposition~3.4]{AS}.  Carlitz gives the following skew-Schur function expansion for the quotient  $ \dfrac{E_{odd}^{alt}}{E_{even}^{alt}}.$
\begin{thm}\label{Carlitz} (Carlitz) Let $\delta_n=(n-1,n-2,\dots, 1), n\ge 2,$ be the staircase shape. (Set $\delta_1=\emptyset.$) Then 
\begin{align} \dfrac{E_{odd}^{alt}}{E_{even}^{alt}}&=s_{(1)}+\sum_{n\ge 3} s_{\delta_n/\delta_{n-2}}\\
&= \tan (\sum_{i\ge 1} \arctan{x_i}). \label{Carlitz-tan}
\end{align}
\end{thm}
Carlitz' proof that  $\dfrac{E_{odd}^{alt}}{E_{even}^{alt}}$ equals  \eqref{Carlitz-tan}   is recovered in \cite[Proof of Proposition~3.4]{AS}, in the  language of symmetric functions and skew tableaux.

Let $E_n$ be the $n$th Euler number; $E_n$ is the cardinality of  the set of alternating (down-up) permutations $\{\sigma\in S_n: \sigma(1)>\sigma(2)<\sigma(3)>\dots\}.$  Note that $E_{2n+1}$ is the tangent number which appeared in Proposition~\ref{myHE1a} and Proposition~\ref{myHE2a}.  Also let $z_\lambda$ denote the order of the centraliser of an element of $S_n$ with cycle-type $\lambda\vdash n;$ thus $z_\lambda=\prod_i i^{m_i} m_i!$ if $\lambda$ has $m_i$ parts equal to $i.$ The following theorem is due to Foulkes \cite{Foulkes}.  See also \cite{StAltEnum}. %
\begin{thm}\label{Foulkes} (Foulkes) Let $\delta_n=(n-1,n-2,\ldots, 1).$ 
Then 
\[ s_{\delta_n/\delta_{n-2}}=\sum_{\stackrel{\lambda\vdash 2n-3}{\lambda \text{ has only odd parts}}} (-1)^{\frac{n-\ell(\lambda)}{2}} E_{\ell(\lambda)} z_\lambda^{-1} p_\lambda,\]
In particular the dimension of the representation indexed by the skew shape $\delta_n/\delta_{n-2}$ is the Euler number $E_n.$
\end{thm}

We now deduce a similar skew-Schur function expansion for the quotient appearing in the left side of Theorem~\ref{OrigRPS}.

\begin{thm}\label{myAltCarlitz} Let $\delta_n=(n-1,n-2,\ldots, 1), n\ge 2.$  (Set $\delta_1=\emptyset.$) Then 
 \begin{align} \dfrac{E_{odd}}{E_{even}}&=s_{(1)}+\sum_{n\ge 3} (-1)^{n} s_{\delta_n/\delta_{n-2}}\\
&= \tanh (\sum_{i\ge 1} \arctanh{x_i}). \label{myarctanh}
\end{align}
 If $H\!k_n$ is the sum of all Schur functions indexed by hooks of size $n,$ then 
 \begin{equation}\label{Hooks}(-1)^{n-2} s_{\delta_n/\delta_{n-2}} =\sum_{\stackrel {\mu\vdash (2n-1)}  {\mu=\prod i^{m_i}}} (-1)^{\ell(\mu)-1}\binom{\ell(\mu)}{m_1, m_2,\ldots}
H\!k_1^{m_1}H\!k_2^{m_2}\ldots 
\end{equation}
\end{thm}

\begin{proof} Rewrite Carlitz' result, Theorem~\ref{Carlitz},  as a generating function, noting that $s_{\delta_n/\delta_{n-2}}$ has degree $1+2(n-2).$ 
\begin{equation}\label{Carlitz-gf} \dfrac{E_{odd}^{alt}(t)}{E_{even}^{alt}(t)}=ts_{(1)}+\sum_{n\ge 3} t\, t^{2(n-2)}s_{\delta_n/\delta_{n-2}}\\
\end{equation}
Clearly the substitution $t\rightarrow it$ converts the left-hand side to $\dfrac{iE_{odd}(t)}{E_{even}(t)}$ and 
the right-hand side to 
\[i \left( ts_{(1)}+\sum_{n\ge 3}t^{2n-3} (-1)^{n-2} s_{\delta_n/\delta_{n-2}}\right).\]

Similarly, rewrite the second equality in Theorem~\ref{Carlitz} as a generating function, giving:
\begin{equation} \label{tgtIdentity}
\dfrac{E^{alt}_{odd}(t)}{E^{alt}_{even}(t)}=\tan(\sum_{j\ge 1} \arctan( tx_j)),
\end{equation}
and make the same substitution $t\rightarrow it,$  
noting  that 
\begin{equation}\label{tan-tanh}\arctan(iz)=i\arctanh(z), \tan (ix)=i\tanh(x).\end{equation}
Alternatively, it is easy to show directly 
that $\mathcal{E}=\dfrac{E_{odd}}{E_{even}}$ satisfies the same recurrence as the hyperbolic tangent function, namely 
\[\tanh(x+y)=\dfrac{\tanh(x)+\tanh(y)}{1+\tanh(x)\tanh(y)}.\]
Indeed, writing $f(x\ge m)$ for the function $f(x_m, x_{m+1},\ldots),$ 
it can be verified that 
\[\mathcal{E}(x\ge 1)=\dfrac{x_1+\mathcal{E}(x\ge 2)}{1+x_1\mathcal{E}(x\ge 2)}.\]

Finally, 
the expression \eqref{Hooks}   follows by using Part (2) of Proposition~\ref{myHE1a}.  
\end{proof}
We observe  two  equivalent identities that are consequences of Propositions~\ref{myHE1a} and ~\ref{myHE2a}, and equations ~\eqref{Carlitz-tan} and ~\eqref{myarctanh}.   By ~\eqref{tan-tanh} the two identities ~\eqref{tan} and ~\eqref{tanh} are equivalent, and also follow from the power series expansion of  the arctangent function.
\begin{cor}
\begin{equation}\label{tan}
 \sum_{j\ge 1} \arctan tx_j= \sum_{k\ge 0} \dfrac{(-1)^k p_{2k+1} t^{2k+1}}{2k+1},
\end{equation}
and hence 
\begin{equation}\label{tan2}
 \left(\sum_{j\ge 1} \arctan x_j\right)\big[Lie_{odd}^{alt}\big]=\arctan p_1.
\end{equation}
\begin{equation}\label{tanh}
 \sum_{j\ge 1} \arctanh tx_j= \sum_{k\ge 0} 
\dfrac{ p_{2k+1} t^{2k+1}}{2k+1},
\end{equation}
and hence 
\begin{equation}\label{tanh2}
\left(\sum_{j\ge 1} \arctanh x_j\right)\left[Lie_{odd}\right]=\arctanh p_1.
\end{equation}
\end{cor}

\begin{remark} From Eqn.~\eqref{BrandtLie}, we see that the odd Lie characteristics $Lie_{odd}$ belong to the $\mathbb{Q}$-subalgebra of symmetric functions generated by the odd power sums $\{p_{2k+1}, k\geq 0\}.$
The expansions for $\tanh(x)$ and $\tan(x)$ in Part (3) of Proposition~\ref{myHE2a} and  Proposition~\ref{myHE1a} respectively, show that the same is true for the quotients $\dfrac{H_{odd}}{H_{even}}=\dfrac{E_{odd}}{E_{even}}$ and $\dfrac{H_{odd}^{alt}}{H_{even}^{alt}}=\dfrac{E_{odd}^{alt}}{E_{even}^{alt}}$, and 
 similarly  for the product $H\!E$, by expanding the exponential in Part (4) of  Proposition~\ref{FundamentalHE} as a product of power series in $p_{2k+1}/(2k+1)$.   See also \cite[Chapter III, Section 8, Eqn. (8.5)]{M} and \cite[Exercise 19]{St5EC2Supp}.   
  A theorem of Ardila and Serrano \cite[Theorem 4.3]{AS} asserts that the staircase Schur function $s_{\delta_n/\delta_{n-2}}$, and hence  the quotients ${E_{odd}^{alt}}/{E_{even}^{alt}}$, have a positive expansion in terms of Schur $P$-functions,  which form a basis for 
$\mathbb{Q}[p_{2k+1}, k\geq 0]$.
\end{remark}

We close by mentioning one other context known to us, in which the odd Lie representations arise. Consider the $S_n$-representation $\eta_n$ on the multilinear component of the free Jordan algebra with $n$ generators.  The Jordan algebra has a bracket defined by $[x,y]=xy+yx.$ The $S_n$-module afforded by this bracket and its deformations were determined in \cite{CHS}. 
The main theorem of \cite{CHS} implies that the $S_n$-representation on the free Jordan algebra is determined by symmetrising $Lie_{odd}$.  Recall that $H=\sum_{n\ge 0}h_n.$
\begin{thm} \cite[Theorem 2.1]{CHS} Define $\eta_0=1.$ The following plethystic identity holds:
\begin{equation*} \sum_{n\ge 0} \eta_n=H[Lie_{odd}]
\end{equation*}
\end{thm}
Clearly this equation can  be decomposed according to parity to give plethystic expressions for $\sum_{n\ge 0, n\equiv i\pmod 2} \eta_n, i=0,1.$

\vskip.15in

\noindent \textbf{Acknowledgments.} 
{I am grateful to Richard Stanley for communicating his conjecture, and for  supplying pertinent references.  
   I also wish to convey my gratitude to the anonymous referees for a careful reading of the paper, and for  comments  leading to valuable  improvements. }

\bibliographystyle{amsplain.bst}

\end{document}